\documentclass[12pt,twoside,reqno]{amsart}
\usepackage{amsmath}
\usepackage{amsfonts}
\usepackage{amssymb}
\usepackage{color}
\usepackage{mathrsfs}
\usepackage{cite}
\usepackage{geometry}
\usepackage{marginnote}
\usepackage{todonotes}
\allowdisplaybreaks
\newcommand{\aone}{{k}}
\newcommand{\atwo}{\alpha}
\definecolor{darkmagenta}{rgb}{0.55, 0.0, 0.55}
\newtheorem{theorem}{Theorem}[section]

\newtheorem{lemma}[theorem]{Lemma}
\newtheorem{proposition}[theorem]{Proposition}
\newtheorem{definition}[theorem]{Definition}
\newtheorem{remark}[theorem]{Remark}
\allowdisplaybreaks
\numberwithin{equation}{section}
\begin{document}
\title{No pattern formation in a quasilinear chemotaxis model with local sensing} 
\thanks{Partially supported by PHC Star~43570WD - Mathematical Modeling and Analysis of Phenomena from Biology}

\author{Philippe Lauren\c{c}ot}
\address{Laboratoire de Math\'ematiques (LAMA), UMR~5127, Universit\'e Savoie Mont Blanc, CNRS \\ F--73000 Chamb\'ery, France}
\email{philippe.laurencot@univ-smb.fr}

\author{Ariane Trescases}
\address{Institut de Math\'ematiques de Toulouse, UMR~5219, Universit\'e de Toulouse, CNRS \\ F--31062 Toulouse Cedex 9, France}
\email{ariane.trescases@math.univ-toulouse.fr}


\date{\today}

\begin{abstract}
Convergence to spatially homogeneous steady states is shown for a chemotaxis model with local sensing and possibly nonlinear diffusion when the intrinsic diffusion rate~$\phi$ dominates the inverse of the chemotactic motility function~$\gamma$, in the sense that $(\phi\gamma)'\ge 0$. This result encompasses and complies with the analysis and numerical simulations performed in Choi \& Kim (2023). The proof involves two steps: first, a Liapunov functional is constructed when $\phi\gamma$ is non-decreasing. The convergence proof relies on a detailed study of the dissipation of the Liapunov functional and requires additional technical assumptions on~$\phi$ and~$\gamma$.
\end{abstract}

\maketitle

%
%
\pagestyle{myheadings}
\markboth{\sc{Ph. Lauren\c cot \& A. Trescases}}{\sc{No pattern formation in a quasilinear chemotaxis model with local sensing}}

\section{Introduction}\label{sec.int}

In \cite{ChKi2023}, Choi \&~Kim study the emergence of patterns in the chemotaxis model with local sensing
\begin{subequations}\label{ks}
	\begin{align}
		& \partial_t u = \Delta\big( \phi(u) \gamma(v)\big), \qquad (t,x)\in (0,\infty)\times \Omega, \label{ks1}\\
		& \partial_t v = \Delta v -v + u , \qquad (t,x)\in (0,\infty)\times \Omega, \label{ks2}\\
		& \nabla\big( \phi(u)\gamma(v)\big)\cdot \mathbf{n} = \nabla v\cdot \mathbf{n} = 0, \qquad (t,x)\in (0,\infty)\times \partial\Omega, \label{ks3}\\
		& (u,v)(0) = (u^{in},v^{in}), \qquad x\in\Omega, \label{ks4}
	\end{align}
\end{subequations}
where $\Omega$ is a smooth bounded domain of $\mathbb{R}^n$, $n\ge 1$, $\mathbf{n}$ denotes the outward unit normal vector field to $\partial\Omega$, $u^{in}$ and $v^{in}$ are non-negative initial conditions, and the functions $\phi$ and $\gamma$  are given by
\begin{subequations}\label{pgck}
\begin{equation}
	\phi(s) = s^m, \quad \gamma(s) = a + \frac{b}{(s+s_0)^k}, \qquad s>0,  \label{pgck1}
\end{equation}
with parameters
\begin{equation}
	m>0, \quad a\ge 0, \quad b>0, \quad k>0, \quad s_0\ge 0. \label{pgck2}
\end{equation}
\end{subequations}
The system~\eqref{ks} describes the dynamics of the density~$u\ge 0$ of a population of cells secreting a chemical or signal with concentration~$v\ge 0$, in the spirit of the class of systems introduced by Keller \&~Segel \cite{KeSe70,KeSe71a}, considered in the context of local sensing (as opposed to gradient sensing, see \cite{DKTY2019,KeSe71a,OtSt97} for discussions in the case $\phi=\mathrm{id}$). The motion of cells is triggered by diffusion and a chemotactic bias, both being mediated by the density of cells~$u$ (when $m\ne 1$) and the chemical concentration~$v$, 
\begin{equation*}
	\partial_t u = \mathrm{div}\big( \phi'(u) \gamma(v) \nabla u + \phi(u) \gamma'(v) \nabla v \big),
\end{equation*}
while the chemical spreads by classical diffusion, thereby resulting in a triangular system of two second-order parabolic equations with a cross-diffusion term involving both variables. Observe that, owing to the choice~\eqref{pgck} of the motility~$\gamma$ with $\gamma'\le 0$, the chemical plays the role of a chemoattractant and the larger the chemical concentration, the weaker the chemotactic bias.

The analysis performed in \cite{ChKi2023} relies, in particular, on a detailed study of the set 
\begin{equation}\label{def:E}
	E := \left\{ s>0\ :\ - \frac{\phi\gamma'}{\phi'\gamma}(s) > 1 \right\} = \left\{ s>0\ :\ \frac{k}{m} \frac{s}{s+s_0} \frac{b}{b+ a(s+s_0)^k} > 1 \right\}
\end{equation}
of excitable density values and the emergence and the shapes of patterns are shown to be intimately connected to the structure of~$E$, a feature supported as well by numerical simulations. In particular, no pattern formation is expected when~$E$ is empty, and the main contribution of this paper is to show that this is indeed the case. Specifically, when~$\phi$ and~$\gamma$ are given by~\eqref{pgck}, we prove the following result.

\begin{theorem}\label{thm0}
	Assume that $\phi$ and $\gamma$ satisfy~\eqref{pgck} with $(\phi \gamma)'\ge 0$. Let $(u,v)$ be a non-negative weak solution to~\eqref{ks} in the sense of Definition~\ref{defin} below with non-negative initial conditions 
	\begin{equation*}
		u^{in}\in L^1(\Omega)\cap H^1(\Omega)', \qquad v^{in}\in C\big(\overline{\Omega}\big),
	\end{equation*}
	satisfying
	\begin{equation*}
		\text{ either }\;\; \left( s_0>0 \; \text{ and } \; m\ge \frac{k}{2} \right) \;\;\text{ or }\;\;  \inf_{\Omega} v^{in} > 0.
	\end{equation*}
	Then, there is $p_0\in [1,2)$ depending only on~$m$ and~$k$ such that, for $p\in [1,p_0]$ and $q\in [1,2)$, 
	\begin{equation}
		\lim_{t\to\infty} \|v(t) - M\|_q = \lim_{t\to\infty} \int_t^{t+1} \|u(\tau)-M\|_p^p\ \mathrm{d}\tau = 0 \label{cv0}
	\end{equation}
	with
	\begin{equation}\label{def:mass}
		M := \frac{1}{|\Omega|} \int_\Omega u^{in}(x)\ \mathrm{d}x
	\end{equation}
	Moreover, if $m>k+1$, then the convergence of $v$ in~\eqref{cv0} holds true for any $q\in [1,m+1-k)$.
\end{theorem}

Theorem~\ref{thm0} is a consequence of Theorems~\ref{thm2} and \ref{thm3} (see Section~\ref{sec:appli}) and extends and improves \cite[Theorem~1.5]{DLTW2023} dealing with the semilinear case $m=1$. The proof of \cite[Theorem~1.5]{DLTW2023} relies on the construction of a Liapunov functional involving the $(H^1)'$-norm of $u$ and the $L^1$-norm of $G_0(v)$ with the function $G_0$ given by $G_0'(s) = 2 s\gamma(s) - \gamma(M) s - M \gamma(M)$ for $s>0$ and $G_0(M)=0$, and requires the function $s\mapsto s\gamma(s)$ to be non-decreasing. Though it does not seem possible to adapt this construction to the quasilinear setting $\phi\ne\mathrm{id}$, we nevertheless succeed in constructing a Liapunov functional for~\eqref{ks} when $\phi\gamma$ is non-decreasing. The construction, performed in  Section~\ref{sec.tie}, is far more tricky and the Liapunov functional we obtain features the $(H^1)'$-norm of $u-v$ and $u$ besides the $L^2$-norm of $v$ and the $L^1$-norm of $\Psi(v)$ with $\psi$ given by $\Psi'=\phi\gamma$ and $\Psi(0)=0$. The next step is to exploit both the Liapunov functional and its dissipation to derive compactness and stabilization properties of the families $\mathcal{U} := \{u(t)\ :\ t\ge 0\}$ and $\mathcal{V} := \{v(t)\ :\ t\ge 0\}$. As the topology for $\mathcal{U}$ is rather weak, additional technical assumptions on $\phi$ and $\psi$ are needed to show the convergence of $u$ and $v$ to the spatially homogeneous steady state $(M,M)$ and three different settings are considered in the next section. In particular, the topology in which the convergence of $u$ takes place depends on these technical assumptions.

Let us now describe the contents of this paper: in the next section, we introduce the general class of functions $\phi$ and $\gamma$ to be handled in this paper and identify three cases in which we can prove long term convergence of weak solutions to spatially homogeneous steady states. We already point out here that we obtain only weak convergence as $t\to\infty$ of $u$ in $L^2(t,t+1;H^1(\Omega)')$ in Theorem~\ref{thm1}, while convergence in $L^1((t,t+1)\times\Omega)$ and $L^p((t,t+1)\times\Omega)$ for some $p\in (1,2)$ is established in Theorem~\ref{thm2} and Theorem~\ref{thm3}, respectively. Section~\ref{sec.tie} is devoted to the derivation of time-independent estimates and the construction of the Liapunov functional. The main property used in this section is the monotonicity of $\phi\gamma$, along with a growth condition on $\phi\gamma$, see~\eqref{hypcomp} below. The long term behaviour is studied in Section~\ref{sec.ltb} and begins with a lemma collecting properties of cluster points of $\{v(t)\ :\ t\ge 0\}$ in $L^1(\Omega)$ as $t\to\infty$. The identification of these cluster points is the step which requires additional assumptions on $\phi$ and $\gamma$. Finally, in Section~\ref{sec:appli}, we come back to the case where $\phi$ and $\gamma$ are given by~\eqref{pgck} and prove Theorem~\ref{thm0}. A short discussion on its outcome ends the paper.

\section{Main results}\label{sec.mr}

Throughout this paper, we assume that the diffusion rate~$\phi$ and the motility~$\gamma$ satisfy the following properties:
\begin{equation}
	\phi\in C([0,\infty))\cap C^1((0,\infty)), \quad \phi(0)=0, \quad \phi>0 \;\;\text{ and }\;\; \phi'>0 \;\;\text{ on }\;\; (0,\infty), \label{hypphia} 
\end{equation}
and
\begin{equation}
	\gamma\in C([0,\infty)) \cap C^1((0,\infty)), \quad \gamma>0 \;\;\text{ on }\;\; (0,\infty). \label{hypgamma}
\end{equation}
We further assume that there is $\kappa>0$ such that
\begin{equation}
	(\phi\gamma)(s) \le \kappa \big( 1 + \Psi(s) \big), \qquad s\ge 0, \label{hypcomp}
\end{equation}
where 
\begin{equation}
	\Psi(s) := \int_0^s (\phi\gamma)(s_*)\ \mathrm{d}s_*\ge 0, \qquad s\ge 0. \label{Psi}
\end{equation}
We note that $\Psi$ is well-defined, non-negative, and non-decreasing on $[0,\infty)$, according to~\eqref{hypphia} and~\eqref{hypgamma}. The condition~\eqref{hypcomp} implies that $\Psi$ grows at most exponentially fast at infinity. Let us mention at this point that the functions $\phi$ and $\gamma$ given by~\eqref{pgck} satisfy~\eqref{hypphia}, \eqref{hypgamma}, and~\eqref{hypcomp} for $s_0>0$. The forthcoming analysis however does not require the monotonicity of $\gamma$.

\medskip

Before state our result, let us define the notion of weak solution to~\eqref{ks} to be used in the sequel. Owing to the possible degeneracy of~\eqref{ks1}, when either $u$ or $\gamma(v)$ vanishes, we do not expect classical solutions to exist.

\begin{definition}[weak solution] \label{defin}
	Consider non-negative initial conditions $u^{in}$ and $v^{in}$ satisfying
	\begin{equation}\label{ci:reg}
		u^{in} \in L^1(\Omega) \cap H^1(\Omega)', \qquad
		v^{in} \in L^2(\Omega), \qquad
		\Psi(v^{in}) \in L^1(\Omega).
	\end{equation}
	A (global) weak solution of~\eqref{ks} is a couple of non-negative functions $(u,v)$ such that, for all $T>0$,
	\begin{align*} 
		u &\in L^\infty(0,T;L^1(\Omega)) \cap L^\infty(0,T;H^1(\Omega)'), \\ 
		v &\in L^\infty(0,T;L^2(\Omega)) \cap H^1(0,T;H^1(\Omega)'),\\
		\phi(u) \gamma(v) &\in L^1((0,T)\times\Omega), \quad \Psi(v) \in L^\infty(0,T;L^1(\Omega)),
	\end{align*}
	which satisfies 
	\begin{equation*}
		\int_\Omega u^{in} \varphi(0) ~\mathrm{d}x - \int_0^\infty \int_\Omega u \partial_t \varphi ~\mathrm{d}x ~\mathrm{d}t - \int_0^\infty \int_\Omega \phi(u)\gamma(v) \, \Delta \varphi ~\mathrm{d}x ~\mathrm{d}t = 0 
	\end{equation*}
	and
	\begin{equation*}
		\int_\Omega v^{in} \varphi(0) ~\mathrm{d}x - \int_0^\infty \int_\Omega v \partial_t \varphi ~\mathrm{d}x ~\mathrm{d}t + \int_0^\infty \int_\Omega v \left(\varphi - \Delta \varphi \right) ~\mathrm{d}x ~\mathrm{d}t = \int_0^\infty \int_\Omega u \varphi ~\mathrm{d}x ~\mathrm{d}t 
	\end{equation*}
	for all $\varphi \in C^2_c([0,\infty)\times\overline{\Omega})$ with $\nabla\varphi\cdot \mathbf{n}=0$ on $(0,\infty)\times\partial\Omega$.
\end{definition}

We shall not discuss here the existence of weak solutions to~\eqref{ks}, postponing this issue to future research. Besides the semilinear case $\phi=\mathrm{id}$ which has been throroughly studied \cite{BLT2020, DKTY2019, DLTW2023, LiJi2020, TaWi2017} and for which classical solutions are also available \cite{FuJi2021a, FuJi2021b, FuSe2022a, FuSe2022b, JLZ2022, LyWa2021, YoKi2017}, weak solutions are constructed in \cite{Wink2020} for $\phi(s)=s^m$, $m>1$, and $\gamma\in C^3((0,\infty))$ satisfying $\gamma>0$ on $(0,\infty)$, along with 
\begin{equation*}
	\gamma\in W^{1,\infty}(s_0,\infty) \;\;\text{ for any}\;\; s_0>0 \;\;\text{ and }\;\; \inf_{s>1}\{ s^\aone\gamma(s)\} >0 \;\;\text{ for some }\; \aone\ge 0,
\end{equation*} 
when $n\ge 2$, $m>n/2$, and 
\begin{equation*}
	\aone < \min\left\{ \frac{nm-2}{n-2} , \frac{2(m+1)}{n-2}\right\}.
\end{equation*}
Related chemotaxis models featuring~\eqref{ks1} with $m>1$ are considered in \cite{Ren2022} and \cite{XuWa2021}.

\medskip

We now turn to the long term behaviour of weak solutions to~\eqref{ks} and first consider the case where $\phi\gamma$ is increasing.

\begin{theorem}[convergence for increasing $\phi \gamma$]\label{thm1}
Consider two functions $\phi$ and $\gamma$ satisfying~\eqref{hypphia}, \eqref{hypgamma}, and~\eqref{hypcomp}, and assume that
\begin{equation*}
	(\phi\gamma)' > 0 \;\text{ on }\; (0,\infty). 
\end{equation*}
Let $(u,v)$ be a global weak solution to~\eqref{ks} in the sense of Definition~\ref{defin} with non-negative initial conditions $(u^{in},v^{in})$ satisfying~\eqref{ci:reg} and
\begin{equation}
	M := \frac{1}{|\Omega|} \int_\Omega u^{in}(x)\ \mathrm{d}x > 0. \label{ci}
\end{equation}
Then, for $q\in [1,2)$, and for any $\vartheta\in H^1(\Omega)$,
\begin{equation}
	\lim_{t\to\infty} \|v(t) - M\|_q = \lim_{t\to\infty} \int_t^{t+1} \int_\Omega \big( u(\tau,x) - M \big) \vartheta(x)\ ~\mathrm{d}x~\mathrm{d}\tau = 0. \label{cv1}
\end{equation}
\end{theorem}

\begin{remark}\label{rem1} Theorem~\ref{thm1} applies to the specific choice~\eqref{pgck} of functions~$\phi$ and~$\gamma$ when $m \ge k>0$ and $s_0>0$, as already observed in numerical simulations by Choi \& Kim \cite{ChKi2023}. We refer to Section~\ref{sec:appli} for more details.
\end{remark}

While pointwise strong convergence of $v$ is achieved in Theorem~\ref{thm1}, we only obtain weak convergence on time averages of $u$. We shall now improve the convergence of $u$, assuming in particular the positivity of the initial condition $v^{in}$.

\begin{theorem}[convergence for initial positive concentrations]\label{thm2}
	Consider two functions $\phi$ and $\gamma$ satisfying~\eqref{hypphia}, \eqref{hypgamma}, and~\eqref{hypcomp}, and assume that for all $\lambda >1$, there exists $\eta_\lambda >1$ such that
		\begin{equation}
		\phi(\lambda s) \ge \eta_\lambda \phi(s), \ s> 0, \;\;\;\text{ and }\;\; (\phi\gamma)' \ge 0 \;\text{ on }\; (0,\infty). \label{nd2}
	\end{equation}
Let $(u,v)$ be a global non-negative weak solution to~\eqref{ks} in the sense of Definition~\ref{defin} with non-negative initial conditions $(u^{in},v^{in})$ satisfying~\eqref{ci:reg}, \eqref{ci}, and
\begin{equation}
	v^{in}\in C(\overline\Omega) \quad \text{with} \quad \inf_{\Omega} v^{in} > 0. \label{civ}
\end{equation}
Then, for $q\in [1,2)$,
\begin{equation*}
	\lim_{t\to\infty}  \|v(t) - M\|_q = \lim_{t\to\infty}  \int_t^{t+1} \|u(\tau)-M\|_1 \ \mathrm{d}\tau = 0. 
\end{equation*}
\end{theorem}

\begin{remark}\label{rem2b}
	Thanks to assumption~\eqref{civ} and to the parabolic equation~\eqref{ks2}, the function $v$ has a uniform positive lower bound $v_*$, which we recall in~\eqref{vstar} below. Therefore, $(u,v)$ is also a solution to~\eqref{ks} with $\gamma$ replaced by any function $\tilde\gamma$ that coincides with $\gamma$ on $[v_*,\infty)$. In particular, we can relax \eqref{hypgamma}, so that the result of Theorem~\ref{thm2} holds true for any positive $\gamma \in C^1((0,\infty))$ without assuming the continuity of $\gamma$ at $s=0$.
\end{remark}

\begin{remark}\label{rem2} Theorem~\ref{thm2} applies to the specific choice~\eqref{pgck} of functions~$\phi$ and~$\gamma$ under the mere condition $(\phi \gamma)'\ge0$, see Section~\ref{sec:appli}. This includes in particular the case $s_0=0$, where the condition $(\phi \gamma)'\ge 0$ translates into $m\ge k>0$. The possibility of allowing $s_0$ to vanish, which corresponds to a singular motility $\gamma$, is clearly due to the assumed positivity~ \eqref{civ} of the initial concentration $v^{in}$, as already pointed out in Remark~\ref{rem2b}. Theorem~\ref{thm2} also applies to the choice $\phi(s)=e^{As}-1$ with $A>0$ and $\gamma(s)=s^{-k}$ with $k\in (0,1]$.
\end{remark}

The final result applies to a more restricted class of functions~$\phi$ and~$\gamma$, as it requires a lower algebraic bound on the slope of $\phi$, as well as an upper algebraic bound on $\frac1\gamma$. It however does not rely on the positivity of $v^{in}$ like Theorem~\ref{thm2}, while still providing a strong convergence of time averages of $u$ but now in $L^p(\Omega)$ for some $p\in (1,2)$.

\begin{theorem}[improved convergence]\label{thm3}
	Consider two functions $\phi$ and $\gamma$ satisfying~\eqref{hypphia}, \eqref{hypgamma}, and~\eqref{hypcomp}, and 
	\begin{equation*}
		(\phi\gamma)' \ge 0 \;\text{ on }\; (0,\infty). 
	\end{equation*}
	Assume that there exist $\theta\ge0$, $\atwo\ge1$, $m>0$, and $\aone> 0$ satisfying
	\begin{equation*}
		2 \atwo \ge \theta+\aone, \qquad \text{and}  \qquad m \ge \frac{\aone}{2}, 
	\end{equation*}
	and $(K_\phi,K_\gamma)\in (0,\infty)^2$ such that
	\begin{equation}
		\frac{\phi(r)-\phi(s)}{r-s}
		\ge K_\phi \left[ \mathbf{1}_{(\frac12,2)}\left(\frac rs\right) \frac{|r-s|^{\atwo-1}}{(1+\max\{r,s\})^{\theta}}  + \mathbf{1}_{[0,\frac12]\cup[2,\infty)}\left(\frac rs\right) |r-s|^{m-1}\right] \label{nd4}
	\end{equation}
	for $(r,s)\in (0,\infty)^2$, $r\ne s$, and
	\begin{equation}
	 \frac1{\gamma(s)}\le K_\gamma ( 1+ s)^{\aone},  \qquad s> 0. \label{gac1}
	 \end{equation}
Let $(u,v)$ be a global non-negative weak solution to~\eqref{ks} in the sense of Definition~\ref{defin} below with non-negative initial conditions $(u^{in},v^{in})$ satisfying~\eqref{ci:reg} and~\eqref{ci}.

Then, for $q\in [1,2)$ and $p\in\left [1,\min\left(2\frac{\atwo+1}{\aone+\theta+2},2\frac{m+1}{\aone+2}\right)\right]$ with $p<2$,
\begin{equation*}
	\lim_{t\to\infty} \|v(t) - M\|_q = \lim_{t\to\infty} \int_t^{t+1} \|u(\tau)-M\|_p^p\ \mathrm{d}\tau = 0. 
\end{equation*}
\end{theorem}

\begin{remark} \label{rem3}
Theorem~\ref{thm3} again applies to the specific choice~\eqref{pgck} of functions~$\phi$ and~$\gamma$, see Section~\ref{sec:appli}. As already mentioned in Remark~\ref{rem2b}, in the case where $v^{in}$ verifies~\eqref{civ} one can relax the continuity assumption of $\gamma$ at $s=0$.
\end{remark}

\paragraph{\bf Notation}

Given $f\in H^1(\Omega)'$, we define
\begin{equation*}
	\langle f \rangle := \frac{1}{|\Omega|} \langle f , 1 \rangle_{(H^1)',H^1}, 
\end{equation*} 
and we note that, if $f\in H^1(\Omega)'\cap L^1(\Omega)$, then
\begin{equation*}
	\langle f \rangle = \frac{1}{|\Omega|} \int_\Omega f(x)\ \mathrm{d}x.
\end{equation*}
For $f\in H^1(\Omega)'$ with $\langle f \rangle =0$, let $\mathcal{K}[f]\in H^1(\Omega)$ be the unique (variational) solution to 
\begin{subequations}\label{opK}
	\begin{equation}
		- \Delta\mathcal{K}[f] = f \;\;\text{ in }\;\; \Omega, \quad \nabla \mathcal{K}[f] \cdot \mathbf{n} = 0 \;\;\text{ on }\;\; \partial\Omega, \label{opKa}
	\end{equation}
	satisfying
	\begin{equation}
		\langle \mathcal{K}[f]  \rangle = 0. \label{opKb}
	\end{equation}
\end{subequations}

\section{Time-independent estimates}\label{sec.tie}

We first compute the time evolution of $\langle u\rangle$ and $\langle v \rangle$ which are both explicit and easy to obtain, as well as a classical duality estimate which takes the form of a differential identity satisfied by $\|\nabla\mathcal{K}[u-M]\|_2^2$.

\begin{lemma}\label{lemb01}
	For $t\ge 0$, 
	\begin{align}
		\langle u(t)\rangle & = M = \langle u^{in}\rangle, \label{b01a} \\
		\langle v(t) \rangle & = M + (\langle v^{in}\rangle - M) e^{-t}. \label{b01b}
	\end{align}
Also, 
\begin{equation}
	\frac{1}{2} \frac{\mathrm{d}}{\mathrm{d}t} \|\nabla P\|_2^2 = \int_\Omega (M-u) \phi(u) \gamma(v)\ \mathrm{d}x, \label{b101}
\end{equation}
where
\begin{equation}
	P := \mathcal{K}[u-M]. \label{auxf1}
\end{equation}
\end{lemma}

\begin{proof}
	The conservation of mass~\eqref{b01a} readily follows from~\eqref{ks1} and~\eqref{ks3}, while integration of~\eqref{ks2} over $\Omega$ gives
	\begin{equation}
		\frac{\mathrm{d}}{\mathrm{d}t} \langle v\rangle + \langle v\rangle = M, \qquad t\ge 0. \label{b102}
	\end{equation}
	We then solve~\eqref{b102} to obtain~\eqref{b01b}.
	
	We next infer from~\eqref{ks1} and~\eqref{auxf1} that 
	\begin{align*}
		\frac{1}{2} \frac{\mathrm{d}}{\mathrm{d}t} \|\nabla P\|_2^2 & = \int_\Omega \nabla P\cdot \nabla\partial_t P \ \mathrm{d}x = - \int_\Omega P \Delta\partial_t P\ \mathrm{d}x \\
		& = \int_\Omega P \partial_t u\ \mathrm{d}x = \int_\Omega \phi(u)\gamma(v) \Delta P\ \mathrm{d}x \\
		& = \int_\Omega (M-u) \phi(u) \gamma(v)\ \mathrm{d}x,
	\end{align*}
as claimed.
\end{proof}

We next construct a functional which turns out to be a Liapunov functional for~\eqref{ks} only when $\langle v^{in} \rangle \ge M$, but still provides valuable information otherwise. 

\begin{lemma}\label{lemb02}
	We set 
	\begin{equation}
		Q := \mathcal{K}[v-\langle v\rangle], \quad R := P-Q = \mathcal{K}[u-M-v+\langle v\rangle], \label{auxf2}
	\end{equation}
and define 
\begin{align*}
	\ell_0(u,v) & := \|\nabla R\|_2^2\ge 0, \quad \ell_1(u,v) := \|v\|_2^2 - |\Omega| \langle v\rangle^2 \ge 0, \quad \ell_2(u,v) := 2 \int_\Omega \Psi(v)\ \mathrm{d}x \ge 0, \\
	d_0(u,v) & := \int_\Omega \gamma(v) \big[ \phi(v) - \phi(u) \big] (v-u)\ \mathrm{d}x \ge 0,  \quad d_1(u,v) := \|\nabla\partial_t Q\|_2^2\ge 0, \\
	d_2(u,v) & := \int_\Omega (\phi\gamma)'(v) |\nabla v|^2\ \mathrm{d}x.
\end{align*}
Introducing
\begin{equation*}
	\mathcal{L}_0(u,v) := \sum_{j=0}^2 \ell_j(u,v) \ge 0, \qquad \mathcal{D}_0(u,v) := \sum_{j=0}^2 d_j(u,v),
\end{equation*}
there holds
	\begin{equation}
		\frac{1}{2} \frac{\mathrm{d}}{\mathrm{d}t} \mathcal{L}_0(u,v) + \mathcal{D}_0(u,v) = (M-\langle v\rangle) \int_\Omega \phi(u) \gamma(v)\ \mathrm{d}x. \label{b100}
	\end{equation}
\end{lemma}

\begin{proof}
We first observe that, by~\eqref{ks2} and~\eqref{auxf2},
\begin{equation*}
	\partial_t \big( - \Delta Q + \langle v \rangle \big) - \Delta v = u - v = - \Delta R + M - \langle v \rangle  = -\Delta R + \partial_t \langle v \rangle \;\;\text{ in }\;\; (0,\infty)\times\Omega.
\end{equation*}
Hence, $- \Delta \big( \partial_t Q+v -R \big) = 0$ in $(0,\infty)\times\Omega$, from which we deduce that 
\begin{equation}
	\partial_t Q + v - R = \langle v \rangle \;\;\text{ in }\;\; (0,\infty)\times\Omega. \label{b02} 
\end{equation}
It follows from~\eqref{ks}, \eqref{auxf2}, and~\eqref{b02} that 
\begin{align*}
	\frac{1}{2} \frac{\mathrm{d}}{\mathrm{d}t} \left( \|\nabla R\|_2^2 + \|v\|_2^2 \right) & = \int_\Omega \left( \nabla R\cdot \nabla\partial_t R + v \partial_t v \right)\ \mathrm{d}x = \int_\Omega \left( - R \Delta\partial_t R + v \partial_t v \right)\ \mathrm{d}x \\
	& = \int_\Omega \left( R \partial_t u - R \partial_t v + R \partial_t \langle v \rangle + v \partial_t v \right)\ \mathrm{d}x \\
	& = \int_\Omega \phi(u) \gamma(v) \Delta R\ \mathrm{d}x + \int_\Omega (v-R) \partial_t v\ \mathrm{d}x \\
	& = \int_\Omega \phi(u) \gamma(v) \left( - u + M + v - \langle v \rangle \right)\ \mathrm{d}x + \int_\Omega \left( \langle v\rangle - \partial_t Q \right) \partial_t v\ \mathrm{d}x \\
	& = \int_\Omega \phi(u) \gamma(v) (v-u)\ \mathrm{d}x + (M - \langle v\rangle) \int_\Omega \phi(u) \gamma(v)\ \mathrm{d}x \\
	& \qquad + \frac{|\Omega|}{2} \frac{\mathrm{d}}{\mathrm{d}t} \langle v\rangle^2 - \int_\Omega \partial_t Q \partial_t (v - \langle v\rangle)\ \mathrm{d}x.
\end{align*}
Since 
\begin{equation*}
	\int_\Omega \partial_t Q \partial_t (v - \langle v\rangle)\ \mathrm{d}x = - \int_\Omega \partial_t Q \Delta\partial_t Q\ \mathrm{d}x = \|\nabla\partial_t Q\|_2^2,
\end{equation*}
we obtain the identity
\begin{equation}
	\begin{split}
	\frac{1}{2} \frac{\mathrm{d}}{\mathrm{d}t} \left( \|\nabla R\|_2^2 + \|v\|_2^2 - |\Omega| \langle v\rangle^2\right) & = \int_\Omega \phi(u) \gamma(v) (v-u)\ \mathrm{d}x \\
	& \qquad + (M - \langle v\rangle) \int_\Omega \phi(u) \gamma(v)\ \mathrm{d}x - \|\nabla\partial_t Q\|_2^2.
	\end{split} \label{b03}
\end{equation}
Next, using~\eqref{ks2},
\begin{align*}
	\int_\Omega \phi(u) \gamma(v) (v-u)\ \mathrm{d}x & = \int_\Omega \big[ \phi(u) - \phi(v) \big] \gamma(v) (v-u)\ \mathrm{d}x + \int_\Omega (\phi \gamma)(v) (v-u)\ \mathrm{d}x \\ 
	& = - \int_\Omega \gamma(v) \big[ \phi(v) - \phi(u) \big] (v-u)\ \mathrm{d}x + \int_\Omega (\phi \gamma)(v) (\Delta v - \partial_t v)\ \mathrm{d}x \\
	& = - \int_\Omega \gamma(v) \big[ \phi(v) - \phi(u) \big] (v-u)\ \mathrm{d}x - \int_\Omega (\phi \gamma)'(v) |\nabla v|^2\ \mathrm{d}x \\
	& \qquad - \frac{\mathrm{d}}{\mathrm{d}t} \int_\Omega \Psi(v)\ \mathrm{d}x.
\end{align*}
Combining~\eqref{b03} and the above inequality, we find~\eqref{b100} and we are left with identifying the signs of the various terms involved in the definition of $\mathcal{L}_0$ and $\mathcal{D}_0$. The non-negativity of $\ell_0$ and $d_1$ is obvious, while that of $\ell_1$ readily follows from Jensen's inequality
\begin{equation*}
	|\Omega| \langle z \rangle^2 \le \|z\|_2^2, \qquad z\in L^2(\Omega),
\end{equation*}
and that of $d_0$ from the monotonicity~\eqref{hypphia} of $\phi$ and the non-negativity~\eqref{hypgamma} of $\gamma$. Finally, the non-negativity of $\ell_2$ is a straightforward consequence of the non-negativity~\eqref{hypphia} and~\eqref{hypgamma} of $\phi$ and $\gamma$ and the definition~\eqref{Psi} of $\Psi$.
\end{proof}

We now split the analysis according to the sign of $M-\langle v^{in}\rangle$ and begin with the case $\langle v^{in}\rangle\ge M$.

\begin{proposition}\label{propb03}
	Assume that $\phi\gamma$ is non-decreasing and that 
	\begin{equation}
		\langle v^{in}\rangle \ge M. \label{b04}
	\end{equation}
Then
\begin{align}
	0 & \le \mathcal{L}_0(u(t),v(t)) \le \mathcal{L}_0(u^{in},v^{in}), \qquad t\ge 0, \label{b08a} \\
	0 & \le \int_0^\infty \mathcal{D}_0(u(t),v(t))\ \mathrm{d}t \le \frac{1}{2} \mathcal{L}_0(u^{in},v^{in}). \label{b08b}
\end{align}
\end{proposition}

\begin{proof}
	We first notice that~\eqref{b01b} and~\eqref{b04} imply that $M-\langle v(t) \rangle \le 0$ for all $t\ge 0$, so that, owing to the non-negativity~\eqref{hypphia} and~\eqref{hypgamma} of $\phi$ and $\gamma$,
	\begin{equation}
		(M - \langle v(t)\rangle) \int_\Omega \phi(u(t,x)) \gamma(v(t,x))\ \mathrm{d}x \le 0, \qquad t\ge 0. \label{b05}
	\end{equation}
Also, the monotonicity of $\phi\gamma$ ensures the non-negativity of $d_2$. Consequently,
\begin{equation*}
	\mathcal{D}_0(u(t),v(t))\ge 0, \qquad t\ge 0.  
\end{equation*}
Gathering~\eqref{b100} and~\eqref{b05} leads us to
\begin{equation*}
	\frac{1}{2} \frac{\mathrm{d}}{\mathrm{d}t} \mathcal{L}_0(u,v) + \mathcal{D}_0(u,v) \le 0, \qquad t\ge 0.
\end{equation*}
Integrating the above differential inequality and using the non-negativity of $\mathcal{L}_0(u,v)$ and $\mathcal{D}_0(u,v)$ immediately give~\eqref{b08a} and~\eqref{b08b}.
\end{proof}

We next turn to the case $\langle v^{in}\rangle < M$ and first observe that the right-hand side of~\eqref{b100} is now non-negative due to the positivity~\eqref{hypphia} and~\eqref{hypgamma} of $\phi$ and $\gamma$ and the positivity of $M-\langle v\rangle $ stemming from~\eqref{b01b}. We actually need to modify $\mathcal{L}_0$ to deal with this additional term.

\begin{proposition}\label{propb04}
	Assume that $\phi\gamma$ is non-decreasing and that 
	\begin{equation}
		\langle v^{in}\rangle < M. \label{b09}
	\end{equation}
Introducing
\begin{align*}
	\ell_3(u,v) & :=  (M-\langle v\rangle) \|\nabla P\|_2^2 \ge 0, \\
	d_3(u,v) & := (M-\langle v\rangle) \int_\Omega u \phi(u) \gamma(v)\ \mathrm{d}x \ge 0, \\
	\mathcal{L}_1(u,v) & := \mathcal{L}_0(u,v) + \ell_3(u,v)\ge 0, \\
	\mathcal{D}_1(u,v) & := \mathcal{D}_0(u,v) + \frac{\ell_3(u,v)}{2} + d_3(u,v) \ge 0,
\end{align*}
there are $C_1>0$ and $C_2>0$ depending only on $\Omega$, $\phi$, $\kappa$, and $M$  such that
	\begin{align}
		0 & \le \mathcal{L}_1(u(t),v(t) \le C_1 \big[ 1 + \mathcal{L}_1(u^{in},v^{in}) \big], \qquad t\ge 0, \label{b10a} \\
		0 & \le \int_0^\infty \mathcal{D}_1(u(t),v(t))\ \mathrm{d}t \le C_2 \big[ 1 +  \mathcal{L}_1(u^{in},v^{in}) \big]. \label{b10b}
	\end{align}
\end{proposition}

\begin{proof} 
	Owing to~\eqref{b01b} and~\eqref{b09},
	\begin{equation}
		M - \langle v(t) \rangle > 0, \qquad t\ge 0, \label{b11}
	\end{equation}
which guarantees the non-negativity of $\ell_3$, as well as that of $d_3$, once combined with the non-negativity of $u$, $\phi$, and $\gamma$. Next, the non-negativity of $\mathcal{L}_1$ and $\mathcal{D}_1$ follows from the already established non-negativity of $\mathcal{L}_0$, $\ell_3$, $\mathcal{D}_0$, and $d_3$.

We then infer from~\eqref{b101} and~\eqref{b102} that 
\begin{align}
	\frac{1}{2} \frac{\mathrm{d}}{\mathrm{d}t}\ell_3(u,v) & = - \frac{1}{2} (M-\langle v\rangle) \|\nabla P\|_2^2 + (M-\langle v\rangle) \int_\Omega (M-u) \phi(u) \gamma(v)\ \mathrm{d}x \nonumber\\
	& = - \frac{\ell_3(u,v)}{2} - (M-\langle v\rangle) \int_\Omega u \phi(u) \gamma(v)\ \mathrm{d}x \label{b12}\\
	& \qquad + M (M-\langle v\rangle) \int_\Omega \phi(u) \gamma(v)\ \mathrm{d}x. \nonumber
\end{align}
Summing~\eqref{b100} and~\eqref{b12}, we obtain
\begin{equation}
	\frac{1}{2} \frac{\mathrm{d}}{\mathrm{d}t} \mathcal{L}_1(u,v) + \mathcal{D}_1(u,v) = (M+1) (M-\langle v\rangle) \int_\Omega \phi(u) \gamma(v)\ \mathrm{d}x. \label{b13}
\end{equation}
Since
\begin{align*}
	\int_\Omega \phi(u) \gamma(v)\ \mathrm{d}x & = \int_\Omega \mathbf{1}_{[0,2(M+1)]}(u) \phi(u) \gamma(v)\ \mathrm{d}x + \int_\Omega \mathbf{1}_{(2(M+1),\infty)}(u) \phi(u) \gamma(v)\ \mathrm{d}x \\
	& \le \phi(2(M+1)) \int_\Omega \mathbf{1}_{[0,2(M+1)]}(u) \gamma(v)\ \mathrm{d}x \\
	& \qquad + \frac{1}{2(M+1)} \int_\Omega \mathbf{1}_{(2(M+1),\infty)}(u) u \phi(u) \gamma(v)\ \mathrm{d}x \\
	& \le \phi(2(M+1)) \int_\Omega \gamma(v)\ \mathrm{d}x + \frac{1}{2(M+1)} \int_\Omega u \phi(u) \gamma(v)\ \mathrm{d}x,
\end{align*}
we deduce from~\eqref{b11} and~\eqref{b13} that
\begin{align*}
	\frac{1}{2} \frac{\mathrm{d}}{\mathrm{d}t} \mathcal{L}_1(u,v) + \mathcal{D}_1(u,v) & \le (M+1)\phi(2(M+1))(M-\langle v\rangle) \int_\Omega \gamma(v)\ \mathrm{d}x \\
	& \qquad + \frac{M-\langle v\rangle}{2} \int_\Omega u \phi(u) \gamma(v)\ \mathrm{d}x \\
	& \le  (M+1)\phi(2(M+1))(M-\langle v\rangle) \|\gamma(v)\|_1 \\
	& \qquad + \frac{\mathcal{D}_1(u,v)}{2}.
\end{align*}
Hence
\begin{equation}
	\frac{\mathrm{d}}{\mathrm{d}t} \mathcal{L}_1(u,v) + \mathcal{D}_1(u,v) \le 2(M+1) \phi(2(M+1))(M-\langle v\rangle) \|\gamma(v)\|_1. \label{b13.5}
\end{equation}
It remains to estimate the $L^1$-norm of $\gamma(v)$. To this end, we notice that~\eqref{hypphia}, \eqref{hypgamma}, and~\eqref{hypcomp} imply that
\begin{align*}
	\gamma(s) \le \sup_{[0,1]}\{\gamma\}, & \qquad s\in [0,1], \\
	\gamma(s) \le \frac{(\phi\gamma)(s)}{\phi(1)} \le \frac{\kappa}{\phi(1)} (1 + \Psi(s)), & \qquad s>1.
\end{align*}
Consequently,
\begin{equation*}
	\|\gamma(v)\|_1 \le \left( \sup_{[0,1]}\{\gamma\} + \frac{\kappa}{\phi(1)} \right) \left( |\Omega| + \|\Psi(v)\|_1 \right), 
\end{equation*}
so that, setting 
\begin{equation*}
	C_0 := (M+1) \phi(2(M+1)) \left( \sup_{[0,1]}\{\gamma\} + \frac{\kappa}{\phi(1)} \right),
\end{equation*}
we infer from~\eqref{b11}, \eqref{b13.5}, and the above inequality that
\begin{align}
	\frac{\mathrm{d}}{\mathrm{d}t} \mathcal{L}_1(u,v) + \mathcal{D}_1(u,v) & \le 2C_0 (M-\langle v\rangle) \left( |\Omega| + \|\Psi(v)\|_1 \right) \nonumber \\
	& \le C_0 (M-\langle v\rangle) \big[ 2|\Omega| + \ell_2(u,v) \big] \nonumber \\
	& \le C_0 (M-\langle v\rangle) \big[ 2|\Omega| + \mathcal{L}_1(u,v) \big] . \label{b14}
\end{align} 
Owing to the non-negativity of $\mathcal{D}_1$ and~\eqref{b01b}, the function $\mathcal{L}_1(u,v)$ satisfies the differential inequality
\begin{equation*}
	\frac{\mathrm{d}}{\mathrm{d}t} \mathcal{L}_1(u(t),v(t))  \le C_0 (M-\langle v^{in}\rangle) e^{-t} \big[ 2|\Omega| + \mathcal{L}_1(u(t),v(t)) \big], \qquad t\ge 0,
\end{equation*}
which gives, after integration,
\begin{equation*}
	2|\Omega| + \mathcal{L}_1(u(t),v(t)) \le \left( 2|\Omega| + \mathcal{L}_1(u^{in},v^{in}) \right) \exp\Big\{ C_0 (M-\langle v^{in}\rangle) \big(1 - e^{-t} \big) \Big\}, \qquad t\ge 0,
\end{equation*}
and from which we deduce~\eqref{b10a} with $C_1 := \big( 1+2|\Omega| \big) e^{C_0 M}$. We next integrate~\eqref{b14} with respect to time and infer from~\eqref{b01b} and~\eqref{b10a} that, for $t>0$, 
\begin{align*}
	\int_0^t \mathcal{D}_1(u(\tau),v(\tau))\ \mathrm{d}\tau & \le \mathcal{L}_1(u(t),v(t)) + \int_0^t \mathcal{D}_1(u(\tau),v(\tau))\ \mathrm{d}\tau \\
	& \le \mathcal{L}_1(u^{in},v^{in}) + C_0 \int_0^t (M-\langle v(\tau)\rangle \big[ 2|\Omega| + \mathcal{L}_1(u(\tau),v(\tau)) \big]\ \mathrm{d}\tau \\
	& \le \mathcal{L}_1(u^{in},v^{in}) + C_0 (M-\langle v^{in}\rangle) \big[ 2|\Omega| + C_1 + C_1 \mathcal{L}_1(u^{in},v^{in}) \big] \int_0^t e^{-\tau}\ \mathrm{d}\tau \\
	& \le C_2 \big[ 1 + \mathcal{L}_1(u^{in},v^{in}) \big],
\end{align*} 
with $C_2 := 1+ C_0 M \big( 2 |\Omega| + C_1 \big)$. Letting $t\to\infty$ in the above estimate gives~\eqref{b10b}.
\end{proof}

Summarizing the outcome of Propositions~\ref{propb03} and~\ref{propb04}, we have derived the following bounds.

\begin{proposition}\label{propb05}
	Assume that $\phi\gamma$ is non-decreasing on $(0,\infty)$. Then there is $C_3>0$ such that
	\begin{subequations}\label{b15}
	\begin{align}
		\sup_{t\in [0,\infty)} \big\{ \|P(t)\|_{H^1} + \|v(t)\|_2 + \|\Psi(v(t))\|_1 \big\} & \le C_3, \label{b15a}\\
		\int_0^\infty \int_\Omega \left[ \gamma(v) (\phi(v)-\phi(u)) (v-u) + (\phi\gamma)'(v) |\nabla v|^2 \right] ~\mathrm{d}x ~\mathrm{d}t & \le C_3, \label{b15b}\\ 
		\int_0^\infty \Big[ \|\partial_t Q(t)\|_{H^1}^2 + \|\partial_t v(t)\|_{(H^1)'}^2 \Big]\ \mathrm{d}t & \le C_3. \label{b15c}
	\end{align}
	\end{subequations}
\end{proposition}

\begin{proof}
Let $t\ge 0$. On the one hand, by~\eqref{b01b}, \eqref{b08a}, and~\eqref{b10a},
\begin{align}
	\|v(t)\|_2^2 & \le \ell_1(u(t),v(t)) + |\Omega| \langle v(t)\rangle^2 \nonumber\\
	& \le \max\big\{ \mathcal{L}_0(u(t),v(t)) , \mathcal{L}_1(u(t),v(t)) \big\} + |\Omega| \max\big\{ M^2 , \langle v^{in}\rangle^2 \big\} \nonumber \\
	& \le \max\big\{ \mathcal{L}_0(u^{in},v^{in}) , C_1 [1+\mathcal{L}_1(u^{in},v^{in})] \big\} + |\Omega| \max\big\{ M^2 , \langle v^{in}\rangle^2 \big\}. \label{b16}
\end{align}
On the other hand, we infer from~\eqref{opK} and the Poincar\'e-Wirtinger inequality
\begin{equation}
	\|z -\langle z\rangle\|_2 \le C_{PW} \|\nabla z\|_2, \qquad z\in H^1(\Omega), \label{pwi}
\end{equation}
that
\begin{align*}
	\|\nabla Q(t)\|_2^2 & = \int_\Omega (Q v)(t,x)\ \mathrm{d}x \le \|Q(t)\|_2 \|v(t)\|_2 \\
	& \le C_{PW} \|\nabla Q(t)\|_2 \|v(t)\|_2 \le \frac{\|\nabla Q(t)\|_2^2}{2} + \frac{C_{PW}^2}{2} \|v(t)\|_2^2.
\end{align*}
Consequently, $\|\nabla Q(t)\|_2 \le C_{PW} \|v(t)\|_2$, from which we deduce that
\begin{equation}
	\begin{split}
	\|P(t)\|_2 & \le C_{PW} \|\nabla P(t)\|_2 \le C_{PW} \big[ \|\nabla R(t)\|_2 + \|\nabla Q(t)\|_2 \big] \\
	& \le C_{PW} \|\nabla R(t)\|_2 + C_{PW}^2 \|v(t)\|_2. 
	\end{split} \label{b17}
\end{equation}
Thanks to~\eqref{b16} and~ \eqref{b17}, the estimates~\eqref{b15a} and~\eqref{b15b} readily follows from Propositions~\ref{propb03} and~\ref{propb04}.

Next, by~\eqref{pwi},
\begin{equation}
		 \|\partial_t Q(t)\|_2^2 \le C_{PW}^2 \|\nabla\partial_t Q(t)\|_2^2. \label{b17.5}
\end{equation}
We also infer from~\eqref{b01b}, \eqref{auxf2}, and H\"older's inequality that, for $t\ge 0$ and $\vartheta\in H^1(\Omega)$, 
\begin{align*}
	\Big| \langle \partial_t v(t) , \vartheta\rangle \Big| & = \left| \left\langle - \Delta\partial_t Q(t) + \frac{\mathrm{d}}{\mathrm{d}t}\langle v(t)\rangle , \vartheta \right\rangle \right| \le \left| \int_\Omega \nabla Q(t)\cdot \nabla\vartheta\ \mathrm{d}x \right| + |\langle v^{in} \rangle -M| \|\vartheta\|_1 e^{-t} \\
	& \le \|\nabla \partial_t Q(t)\|_2 \|\nabla\vartheta\|_2 + |\langle v^{in} \rangle -M| \sqrt{|\Omega|} \|\vartheta\|_2 e^{-t} \\
	& \le \left( \|\nabla \partial_t Q(t)\|_2 + |\langle v^{in} \rangle -M| \sqrt{|\Omega|} e^{-t} \right) \|\vartheta\|_{H^1} .
\end{align*}
Therefore,
\begin{equation*}
	\|\partial_t v(t) \|_{(H^1)'} \le \|\nabla \partial_t Q(t)\|_2 + |\langle v^{in} \rangle -M| \sqrt{|\Omega|} e^{-t},
\end{equation*}
from which we deduce that
\begin{equation}
	\|\partial_t v(t) \|_{(H^1)'}^2 \le 2 \|\nabla \partial_t Q(t)\|_2^2 + 2 |\langle v^{in} \rangle -M|^2 |\Omega| e^{-2t}. \label{b18.5}
\end{equation}
Collecting~\eqref{b17.5} and~\eqref{b18.5} and using~\eqref{b08b} and~\eqref{b10b} give~\eqref{b15c}.
\end{proof}

\section{Large time behaviour}\label{sec.ltb}

We now begin the study of the large time behaviour and first collect some information about the compactness of $(P,Q,v)$ and their cluster points in suitable topologies as $t\to\infty$.

\begin{lemma}\label{lemb06}
Assume that $\phi\gamma$ is non-decreasing on $(0,\infty)$. Then $\mathcal{V} = \{v(t)\ :\ t\ge 0\}$ is relatively compact in $L^q(\Omega)$ for $q\in [1,2)$. 

Consider now a cluster point $v_\infty\in L^1(\Omega)$ of $\mathcal{V}$ in $L^1(\Omega)$ as $t\to\infty$; that is, there is a sequence $(t_j)_{j\ge 1}$ in $(1,\infty)$ such that
\begin{equation}
	\lim_{j\to\infty} t_j = \infty \;\;\text{ and }\;\; \lim_{j\to\infty} \|v(t_j) - v_\infty\|_1 = 0. \label{b20} 
\end{equation}
Then 
\begin{equation}
	v_\infty \in L^2(\Omega) \;\;\text{ with }\;\; \langle v_\infty \rangle = M \;\;\text{ and }\;\; \lim_{j\to\infty} \|v(t_j) - v_\infty\|_q = 0 \label{b20.5}
\end{equation} 
for all $q\in[1,2)$. Moreover, introducing $Q_\infty := \mathcal{K}[v_\infty-M]$ and
\begin{equation*}
	(u_j,v_j,P_j,Q_j)(\tau,x) := (u,v,P,Q)(t_j+\tau,x), \qquad (\tau,x)\in [0,1]\times\Omega, \ j\ge 1,
\end{equation*}
there holds, after possibly extracting a subsequence,
\begin{align}
	& \lim_{j\to\infty} \sup_{\tau\in [0,1]} \|v_j(\tau)-v_\infty\|_q = \lim_{j\to\infty} \sup_{\tau\in [0,1]} \|Q_j(\tau)-Q_\infty\|_2 = 0, \label{b18} \\
	& \lim_{j\to\infty} \int_0^1 \|\nabla Q_j(\tau) - \nabla Q_\infty\|_2^2\ \mathrm{d}\tau = \lim_{j\to\infty} \int_0^1 \|P_j(\tau) - Q_\infty - v_\infty + M \|_q^2\ \mathrm{d}\tau = 0, \label{b18.2}
 \end{align}
 for any $q\in [1,2)$.
\end{lemma}

\begin{proof}
We first recall that~\eqref{ks2}, \eqref{ks3}, \eqref{b01a}, and the regularizing properties of the heat semigroup imply that, for each $q\in [1,n/(n-1))$, there is $C_4(q)>0$ such that
\begin{equation}
	\|v(t)\|_{W^{1,q}} \le C_4(q), \qquad t\ge 0. \label{c4}
\end{equation}
In particular, owing to the compactness of the embedding of $W^{1,1}(\Omega)$ in $L^1(\Omega)$ and the $L^2$-bound~\eqref{b15a},
\begin{equation}
	\mathcal{V} = \{v(t)\ :\ t\ge 0\} \;\;\text{ is relatively compact in }\;\; L^q(\Omega) \;\;\text{ for each }\;\; q\in [1,2). \label{b19.5}
\end{equation}

We now consider a cluster point $v_\infty\in L^1(\Omega)$ of $\mathcal{V}$ in $L^1(\Omega)$ as $t\to\infty$ satisfying~\eqref{b20}, its existence being guaranteed by~\eqref{b19.5}. The property~\eqref{b20.5} is then an immediate consequence of~\eqref{b01b}, \eqref{b15a}, and~\eqref{b20}. Next, according to~\eqref{ks}, $v_j$ solves
\begin{equation}
	\begin{split}
		\partial_\tau v_j & = \Delta v_j - v_j + u_j, \qquad (\tau,x)\in (0,1)\times \Omega, \\
		\nabla v_j\cdot \mathbf{n} & = 0, \qquad (\tau,x)\in (0,1)\times \partial\Omega,
	\end{split} \label{b21}
\end{equation}
and
\begin{equation*}
	v_j(0) = v(t_j) 
\end{equation*}
for $j\ge 1$. Let $q\in (1,n/(n-1))\cap (1,2)$, $\vartheta\in W^{n,q'}(\Omega)$ with $q'=q/(q-1)$, and recall that $W^{n,q'}(\Omega)$ embeds continuously in $L^\infty(\Omega)$. It follows from~\eqref{b01a}, \eqref{c4}, \eqref{b21}, and H\"older's inequality that, for $\tau\in (0,1)$,
\begin{align*}
	\left| \langle \partial_\tau v_j(\tau) , \vartheta \rangle \right| & \le \int_\Omega |\nabla v_j(\tau)| |\nabla\vartheta|\ \mathrm{d}x + \|v_j(\tau)\|_q \|\vartheta\|_{q'} + \|u_j(\tau)\|_1 \|\vartheta\|_\infty \\
	& \le \|\nabla v_j(\tau)\|_q \|\nabla \vartheta\|_{q'} +\|v_j(\tau)\|_q \|\vartheta\|_{q'} + C(q) |\Omega| M \|\vartheta\|_{W^{n,q'}} \\
	& \le C(q) \|\vartheta\|_{W^{n,q'}}.
\end{align*}
Therefore,
\begin{equation}
	\left( \partial_t v_j \right)_{j\ge 1} \;\;\text{ is bounded in }\;\; L^\infty((0,1),W^{n,q'}(\Omega)'). \label{b23}
\end{equation}
Since $W^{1,q}(\Omega)$ is compactly embedded in $L^q(\Omega)$ and $L^q(\Omega)$ embeds continuously in $W^{n,q'}(\Omega)'$, we infer from~\eqref{c4}, \eqref{b23}, and \cite[Corollary~4]{Simo1987} that
\begin{equation*}
	(v_j)_{j\ge 1} \;\;\text{ is relatively compact in }\;\; C([0,1],L^q(\Omega)).
\end{equation*}
There are thus a subsequence of $(v_j)_{j\ge 1}$ (not relabeled) and $w\in C([0,1],L^q(\Omega))$ such that
\begin{equation}
	\lim_{j\to\infty} \sup_{\tau\in [0,1]} \|(v_j-w)(\tau)\|_q = 0. \label{b24}
\end{equation}
We now use H\"older's inequality to obtain, for $\vartheta\in W^{1,q'}(\Omega)$, $j\ge 1$, and $\tau\in [0,1]$,
\begin{align*}
	& \left| \int_\Omega \big( w(\tau) - v_\infty \big) \vartheta\ \mathrm{d}x \right| \\
	& \qquad \le \left| \int_\Omega \big( w-v_j \big)(\tau) \vartheta\ \mathrm{d}x \right| + \left| \int_\Omega \big( v_j(\tau) - v_j(0) \big) \vartheta\ \mathrm{d}x \right| + \left| \int_\Omega \big( v_j(0) - v_\infty \big) \vartheta\ \mathrm{d}x \right|\\
	& \qquad \le \|\big( w - v_j\big)(\tau) \|_{q} \|\vartheta\|_{q'} + \left| \int_0^\tau \int_\Omega \partial_\tau v_j(\tau_*) \vartheta\ \mathrm{d}x\mathrm{d}\tau_* \right| + \|v_j(0) - v_\infty \|_{q} \|\vartheta\|_{q'}\\
	& \qquad \le \|\big( w - v_j\big)(\tau) \|_{q} \|\vartheta\|_{q'} + \int_0^\tau  \|\partial_\tau v_j(\tau_*)\|_{(H^1)'} \|\vartheta\|_{H^1}\ \mathrm{d}\tau_* + \|v(t_j) - v_\infty \|_{q} \|\vartheta\|_{q'}\\
	& \qquad \le \left( \|\big( w - v_j\big)(\tau) \|_{q} + \sqrt{\tau} \left( \int_0^\tau \|\partial_\tau v_j(\tau_*)\|_{(H^1)'}^2\ \mathrm{d}\tau_* \right)^{1/2} + \|v(t_j) - v_\infty \|_{q} \right) \|\vartheta\|_{W^{1,q'}}.
\end{align*}
We may then take the limit $j\to\infty$ in the above inequality and deduce from~\eqref{b15c}, \eqref{b20.5}, and~\eqref{b24} that
\begin{equation*}
	\int_\Omega \big( w(\tau) - v_\infty \big) \vartheta(x)\ \mathrm{d}x = 0, \qquad \vartheta\in W^{1,q'}(\Omega), \quad \tau\in [0,1].
\end{equation*}
Consequently,
\begin{equation*}
	w(\tau) = v_\infty, \qquad \tau\in [0,1], 
\end{equation*}
so that the convergence~\eqref{b24} becomes
\begin{equation*}
	\lim_{j\to\infty} \sup_{\tau\in [0,1]} \|v_j(\tau) - v_\infty\|_q = 0. 
\end{equation*}
Recalling the $L^2$-bound~\eqref{b15a}, we have actually established that, for all $q\in [1,2)$,
\begin{equation}
		\lim_{j\to\infty} \sup_{\tau\in [0,1]} \|v_j(\tau) - v_\infty\|_q = 0; \label{c16}
\end{equation}
that is, the first statement in~\eqref{b18}. 

We now turn to the proof of the second statement in~\eqref{b18} and first observe that the definition of $Q$, $Q_j$, and $Q_\infty$ entails that $Q_j-Q_\infty$ is a variational solution to
\begin{align*}
	\partial_\tau (Q_j-Q_\infty) - \Delta (Q_j-Q_\infty) & = \partial_\tau Q_j + v_j-v_\infty -\langle v_j\rangle + M, \qquad (\tau,x)\in (0,1)\times\Omega, \\
	\nabla (Q_j-Q_\infty)\cdot \mathbf{n} & = 0, \qquad (\tau,x)\in (0,1)\times\partial\Omega, \\
	Q_j(0)-Q_\infty & = Q(t_j) - Q_\infty, \qquad x\in \Omega.
\end{align*}
Let $p\in (2n/(n+2),2)$. We infer from the above equations, the continuous embedding of $H^1(\Omega)$ in $L^{p/(p-1)}(\Omega)$, H\"older's inequality, and the Poincar\'e-Wirtinger inequality~\eqref{pwi} that
\begin{align*}
	& \frac{1}{2} \frac{\mathrm{d}}{\mathrm{d}t} \|Q_j-Q_\infty\|_2^2 + \|\nabla (Q_j-Q_\infty)\|_2^2 \\
	& \qquad \le \|Q_j-Q_\infty\|_2 \|\partial_\tau Q_j\|_2 + \|Q_j-Q_\infty\|_{p/(p-1)} \|v_j-v_\infty -\langle v_j\rangle + M\|_{p} \\
	& \qquad \le C_{PW} \|\nabla (Q_j-Q_\infty)\|_2 \|\partial_\tau Q_j\|_2 + C \|\nabla (Q_j-Q_\infty)\|_2 \|v_j-v_\infty\|_p.
\end{align*}
We next use Young's inequality to obtain
\begin{equation*}
	\frac{1}{2} \frac{\mathrm{d}}{\mathrm{d}t} \|Q_j-Q_\infty\|_2^2 + \frac{1}{2} \|\nabla (Q_j-Q_\infty)\|_2^2 \le C \left( \|\partial_\tau Q_j\|_2^2 + \|v_j-v_\infty\|_p^2 \right).
\end{equation*}
Hence, after integrating with respect to time,
\begin{equation}
	\begin{split}
		& \sup_{\tau\in [0,1]} \|Q_j(\tau)-Q_\infty\|_2^2 + \int_0^1 \|\nabla (Q_j(\tau)-Q_\infty)\|_2^2\ \mathrm{d}\tau \\
		& \qquad \le \|Q(t_j)-Q_\infty\|_2^2 + C \left( \int_0^1 \|\partial_\tau Q_j(\tau)\|_2^2\ \mathrm{d}\tau + \sup_{\tau\in [0,1]} \|v_j(\tau)-v_\infty\|_p^2 \right).
	\end{split} \label{b25}
\end{equation}
On the one hand, it follows from the continuous embedding of $H^1(\Omega)$ in $L^{p/(p-1)}(\Omega)$, H\"older's inequality, and the Poincar\'e-Wirtinger inequality~\eqref{pwi} that
\begin{equation*}
	\|Q(t_j)-Q_\infty\|_2 \le C_{PW} \|\nabla(Q(t_j)-Q_\infty)\|_2
\end{equation*}
and
\begin{align*}
	\|\nabla(Q(t_j)-Q_\infty)\|_2^2 & = \int_\Omega \big( Q(t_j)-Q_\infty \big) \big( v(t_j)-\langle v(t_j)\rangle - v_\infty + M \big)\ \mathrm{d}x \\
	& = \int_\Omega \big( Q(t_j)-Q_\infty \big) \big( v(t_j) - v_\infty \big)\ \mathrm{d}x \\
	& \le \| Q(t_j)-Q_\infty\|_{p/(p-1)} \|v(t_j) - v_\infty\|_p \\
	& \le C \| \nabla (Q(t_j)-Q_\infty)\|_{2} \|v(t_j) - v_\infty\|_p.
\end{align*}
Consequently,
\begin{equation*}
	\|Q(t_j)-Q_\infty\|_2 \le C_{PW} \|\nabla(Q(t_j)-Q_\infty)\|_2 \le C \|v(t_j) - v_\infty\|_p
\end{equation*}
and we deduce from~\eqref{b20.5} that
\begin{equation}
	\lim_{j\to\infty} \|Q(t_j) - Q_\infty \|_{H^1} = 0. \label{b26}
\end{equation}
On the other hand, since $p\in (1,2)$, the integrability~\eqref{b15c} and the convergence~\eqref{c16} imply that
\begin{equation}
	 \lim_{j\to\infty} \left[ \int_0^1 \|\partial_\tau Q_j(\tau)\|_2^2\ \mathrm{d}\tau + \sup_{\tau\in [0,1]} \|v_j(\tau)-v_\infty\|_p^2 \right] = 0. \label{b27}
\end{equation}
Combining the above statement with~\eqref{b25} and~\eqref{b26} completes the proof of~\eqref{b18} and proves the first statement in~\eqref{b18.2}.

Finally, let $q\in [1,2)$. We infer from~\eqref{b02} and H\"older's inequality that
\begin{align*}
	& \int_0^1 \|P_j(\tau) - Q_\infty - v_\infty + M\|_q^2\ \mathrm{d}\tau \\
	& \qquad = \int_0^1 \|\partial_\tau Q_j(\tau) + Q_j(\tau) + v_j(\tau) - \langle v_j(\tau) \rangle - Q_\infty - v_\infty + M \|_q^2\ \mathrm{d}\tau \\
	& \qquad \le 4 |\Omega|^{(2-q)/q} \int_0^1 \|\partial_\tau Q_j(\tau)\|_2^2\ \mathrm{d}\tau + 4 |\Omega|^{(2-q)/q} \sup_{\tau\in [0,1]} \|Q_j(\tau) - Q_\infty\|_2^2 \\
	& \qquad\qquad + 8 \sup_{\tau\in [0,1]} \|v_j(\tau) - v_\infty \|_q^2.
\end{align*}
In view of~\eqref{b18} and~\eqref{b27}, the right-hand side of the above inequality converges to zero as $j\to\infty$ and the proof of Lemma~\ref{lemb06} is complete.
\end{proof}

\begin{proof}[Proof of Theorem~\ref{thm1}]
According to Lemma~\ref{lemb06}, $\mathcal{V} = \{v(t)\ :\ t\ge 0\}$ is relatively compact in $L^1(\Omega)$ and we are left with identifying the cluster points of $\mathcal{V}$ in $L^1(\Omega)$ as $t\to\infty$. 
	
We thus consider a cluster point $v_\infty\in L^1(\Omega)$ of $\mathcal{V}$ in $L^1(\Omega)$ as $t\to\infty$ satisfying~\eqref{b20} and keep the notation introduced in Lemma~\ref{lemb06}. For $\varepsilon\in (0,1)$, we deduce from the positivity and continuity of $(\phi\gamma)'$ on $[\varepsilon,1/\varepsilon]$ that
\begin{equation*}
	\mu_\varepsilon := \min_{[\varepsilon,1/\varepsilon]}\{(\phi\gamma)'\}>0.
\end{equation*}
Introducing 
\begin{equation*}
	v_j^\varepsilon := \min\left\{ \max\big\{ v_j , \varepsilon\big\} , \frac{1}{\varepsilon} \right\}, \qquad j\ge 1,
\end{equation*}
which satisfies
\begin{equation*}
	\nabla v_j^\varepsilon = \mathbf{1}_{[\varepsilon,1/\varepsilon]}(v_j) \nabla v_j \;\;\text{ a.e. in }\;\; (0,1)\times\Omega,
\end{equation*}
we infer from the Poincar\'e-Wirtinger inequality that
\begin{align*}
	& \int_0^1 \|v_j(\tau) - \langle v_j(\tau) \rangle \|_1\ \mathrm{d}\tau \\
	& \qquad \le \int_0^1 \|v_j^\varepsilon(\tau) - \langle v_j^\varepsilon(\tau) \rangle \|_1\ \mathrm{d}\tau + \int_0^1 \|(v_j-v_j^\varepsilon)(\tau) - \langle (v_j-v_j^\varepsilon)(\tau) \rangle \|_1\ \mathrm{d}\tau \\
	& \qquad \le C \int_0^1 \|\nabla v_j^\varepsilon(\tau)\|_1\ \mathrm{d}\tau + 2 \sup_{\tau\in [0,1]} \|(v_j-v_j^\varepsilon)(\tau)\|_1 \\
	& \qquad \le \frac{C}{\sqrt{\mu_\varepsilon}} \int_0^1 \int_\Omega \sqrt{\mu_\varepsilon} \mathbf{1}_{[\varepsilon,1/\varepsilon]}(v_j(\tau,x)) |\nabla v_j(\tau,x)|\ \mathrm{d}x\mathrm{d}\tau + 2 \sup_{\tau\in [0,1]} \|\mathbf{1}_{(0,\varepsilon)}(v_j(\tau)) v_j(\tau)\|_1 \\
	& \qquad\qquad + 2 \sup_{\tau\in [0,1]} \|\mathbf{1}_{(1/\varepsilon,\infty)}(v_j(\tau)) v_j(\tau)\|_1 \\
	& \qquad \le \frac{C}{\sqrt{\mu_\varepsilon}} \int_0^1 \int_\Omega \sqrt{(\phi\gamma)'(v_j(\tau,x))} |\nabla v_j(\tau,x)|\ \mathrm{d}x\mathrm{d}\tau + 2 |\Omega|\varepsilon + 2 \varepsilon \sup_{\tau\in [0,1]} \|v_j(\tau)\|_2^2.
\end{align*}
Owing to~\eqref{b15a} and H\"older's inequality, we further obtain that
\begin{align*}
	& \int_0^1 \|v_j(\tau) - \langle v_j(\tau) \rangle \|_1\ \mathrm{d}\tau \\
	& \qquad \le \frac{C \sqrt{|\Omega|}}{\sqrt{\mu_\varepsilon}} \left[ \int_0^1 \int_\Omega (\phi\gamma)'(v_j(\tau,x)) |\nabla v_j(\tau,x)|^2\ \mathrm{d}x\mathrm{d}\tau \right]^{1/2} + 2 |\Omega|\varepsilon + 2 C_3^2 \varepsilon \\
	& \qquad \le \frac{C \sqrt{|\Omega|}}{\sqrt{\mu_\varepsilon}} \left[ \int_{t_j}^\infty \int_\Omega (\phi\gamma)'(v(t,x)) |\nabla v(t,x)|^2\ \mathrm{d}x\mathrm{d}t \right]^{1/2} + 2 |\Omega|\varepsilon + 2 C_3^2 \varepsilon.
\end{align*}
Therefore, owing to~\eqref{b01b}, \eqref{b15b}, \eqref{b20}, and~\eqref{b18},
\begin{equation*}
	\|v_\infty-M\|_1 = \int_0^1 \|v_\infty-M\|_1\ \mathrm{d} \tau = \lim_{j\to\infty} \int_0^1 \|v_j(\tau) - \langle v_j(\tau) \rangle \|_1\ \mathrm{d}\tau \le 2 |\Omega|\varepsilon + 2 C_3^2 \varepsilon.
\end{equation*}
The above inequality being valid for all $\varepsilon\in (0,1)$, we conclude that $v_\infty=M$ a.e. in $\Omega$.

We have thus shown that the constant state $M$ is the only possible cluster point of the family $\mathcal{V} = \{ v(t)\ :\ t\ge 0\}$ in $L^1(\Omega)$ as $t\to\infty$. Together with the compactness of $\mathcal{V}$ in $L^1(\Omega)$ already established in Lemma~\ref{lemb06}, this property implies the convergence of $v(t)$ to $M$ in $L^1(\Omega)$ as $t\to\infty$. Since $\mathcal{V}$ is a bounded subset of $L^2(\Omega)$ according to~\eqref{b15a}, we conclude that
\begin{equation}
	\lim_{t\to\infty} \|v(t)-M\|_q = 0 \;\;\text{ for all }\;\; q\in [1,2). \label{b28}
\end{equation}
Furthermore, it follows from~\eqref{ks2} that, for $t>0$ and $\vartheta\in C_c^2(\Omega)$, 
\begin{align*}
	& \int_t^{t+1} \int_\Omega \big( u(\tau,x) - M \big) \vartheta(x)\ ~\mathrm{d}x ~\mathrm{d}\tau \\
	& \qquad = \int_t^{t+1} \int_\Omega \left[ \partial_t \big( v(\tau,x) - \langle v(\tau)\rangle \big) - \Delta v(\tau,x) + v(\tau,x) - \langle v(\tau) \rangle \right] \vartheta(x)\ ~\mathrm{d}x ~\mathrm{d}\tau \\
	& \qquad =  \int_t^{t+1} \int_\Omega \left[ - \Delta \big( \partial_t Q + v \big)(\tau,x) + v(\tau,x) - \langle v(\tau) \rangle \right] \vartheta(x)\ ~\mathrm{d}x ~\mathrm{d}\tau \\
	& \qquad = \int_t^{t+1} \int_\Omega \partial_t \nabla Q(\tau,x)\cdot \nabla\vartheta(x)\ ~\mathrm{d}x ~\mathrm{d}\tau + \int_t^{t+1} \int_\Omega \big( v(\tau,x) - \langle v(\tau) \rangle \big) (\vartheta - \Delta \vartheta)(x)\ ~\mathrm{d}x ~\mathrm{d}\tau.
\end{align*}
Hence,  by H\"older's inequality,
\begin{align*}
	& \left| \int_t^{t+1} \int_\Omega \big( u(\tau,x) - M \big) \vartheta(x)\ ~\mathrm{d}x ~\mathrm{d}\tau \right| \\
	& \qquad \le \|\vartheta\|_{H^1} \left( \int_t^{t+1} \|\partial_t Q(\tau)\|_{H^1}^2\ \mathrm{d}\tau \right)^{1/2} \\
	& \qquad\qquad + \|\vartheta - \Delta\vartheta\|_\infty \sup_{\tau\in [t,t+1]} \left\{ \|v(\tau) - M\|_1 + \big| M - \langle v(\tau) \rangle \big| \right\}, 
\end{align*}
and we use~\eqref{b01b}, \eqref{b15c}, and~\eqref{b28} to pass to the limit $t\to\infty$ and obtain
\begin{equation*}
	\lim_{t\to\infty} \int_t^{t+1} \int_\Omega \big( u(\tau,x) - M \big) \vartheta(x)\ ~\mathrm{d}x ~\mathrm{d}\tau = 0.
\end{equation*}
A density argument and~\eqref{b15a} allow us to extend the above limit to all $\vartheta\in H^1(\Omega)$ and complete the proof.
\end{proof}

\begin{proof}[Proof of Theorem~\ref{thm2}]
Since $v^{in}$ is positive and continuous in $\overline{\Omega}$ by~\eqref{civ}, it follows from~\eqref{b01a} and \cite[Lemma~2.6]{Fuji2016} that there is $v_*>0$ such that
\begin{equation} 
	v(t,x) \ge v_*>0, \qquad (t,x)\in [0,\infty)\times \Omega. \label{vstar}
\end{equation}
Owing to the positivity and monotonicity of $\phi\gamma$, we further obtain that
\begin{equation}
	(\phi\gamma)(v)(t,x) \ge \omega_* := (\phi\gamma)(v_*)>0, \qquad (t,x)\in [0,\infty)\times \Omega. \label{d1}
\end{equation}

For $\varepsilon\in (0,1)$ and $j\ge 1$, we define 
\begin{align*}
	A_{j,\varepsilon} & := \big\{ (\tau,x)\in (0,1)\times\Omega\ :\ u_j(\tau,x) \ge (1+\varepsilon) v_j(\tau,x) \big\}, \\
	B_{j,\varepsilon} & := \big\{ (\tau,x)\in (0,1)\times\Omega\ :\ u_j(\tau,x) \le (1-\varepsilon) v_j(\tau,x) \big\}, \\
	I_{j,\varepsilon} & := \big\{ (\tau,x)\in (0,1)\times\Omega\ :\ (1-\varepsilon) v_j(\tau,x) < u_j(\tau,x) < (1+\varepsilon) v_j(\tau,x) \big\}, 
\end{align*}
and note that $(0,1)\times\Omega = A_{j,\varepsilon} \cup I_{j,\varepsilon} \cup B_{j,\varepsilon}$. 

On the one hand, on $A_{j,\varepsilon}$, $u_j\ge v_j$ and we infer from~\eqref{nd2}, \eqref{d1}, and the monotonicity of $\phi$ that
\begin{align*}
	\gamma(v_j) \big( \phi(u_j)-\phi(v_j) \big) (u_j-v_j) & \ge \gamma(v_j) \big( \phi((1+\varepsilon)v_j)-\phi(v_j) \big) (u_j-v_j) \\
	& \ge (\eta_{1+\varepsilon}-1) (\phi\gamma)(v_j) (u_j-v_j) \\
	& \ge (\eta_{1+\varepsilon}-1) \omega_* |u_j-v_j|.
\end{align*}
Consequently, thanks to~\eqref{b15b}, 
\begin{equation}
\begin{split}
	& \lim_{j\to\infty} \int_{A_{j,\varepsilon}} |u_j-v_j|\ ~\mathrm{d}x ~\mathrm{d}\tau \\
	& \qquad \le \frac{1}{(\eta_{1+\varepsilon}-1) \omega_*} \lim_{j\to\infty} \int_{A_{j,\varepsilon}} \gamma(v_j) \big( \phi(u_j)-\phi(v_j) \big) (u_j-v_j)\ ~\mathrm{d}x ~\mathrm{d}\tau \\
	& \qquad \le \frac{1}{(\eta_{1+\varepsilon}-1) \omega_*} \lim_{j\to\infty} \int_0^1 \int_\Omega \gamma(v_j) \big( \phi(u_j)-\phi(v_j) \big) (u_j-v_j)\ ~\mathrm{d}x ~\mathrm{d}\tau = 0.
\end{split}\label{d2}
\end{equation}
Similarly, on $B_{j,\varepsilon}$, $u_j\le v_j$ and we infer from~\eqref{nd2}, \eqref{d1}, and the monotonicity of $\phi$ that
\begin{align*}
	\gamma(v_j) \big( \phi(u_j)-\phi(v_j) \big) (u_j-v_j) & = \gamma(v_j) \big( \phi(v_j)-\phi(u_j) \big) (v_j-u_j) \\
	& \ge \gamma(v_j) \big( \phi(v_j) - \phi((1-\varepsilon)v_j) \big) (v_j-u_j) \\
	& \ge \left( 1 - \frac{1}{\eta_{1/(1-\varepsilon)}} \right) (\phi\gamma)(v_j) (v_j-u_j) \\
	& \ge \frac{\eta_{1/(1-\varepsilon)}-1}{\eta_{1/(1-\varepsilon)}} \omega_* |u_j-v_j|.
\end{align*}
We then argue as above to conclude that
\begin{equation}
	\lim_{j\to\infty} \int_{B_{j,\varepsilon}} |u_j-v_j|\ ~\mathrm{d}x ~\mathrm{d}\tau = 0. \label{d3}
\end{equation}
On the other hand, $-\varepsilon v_j \le u_j - v_j \le \varepsilon v_j$ on $I_{j,\varepsilon}$, so that
\begin{equation}
	\int_{I_{j,\varepsilon}} |u_j-v_j|\ ~\mathrm{d}x ~\mathrm{d}\tau \le \varepsilon \int_0^1 \int_\Omega v_j\ ~\mathrm{d}x ~\mathrm{d}\tau \le \varepsilon |\Omega| \max\{ M , \langle v^{in} \rangle\} \label{d4}
\end{equation}
by~\eqref{b01b}. 

Gathering~\eqref{d2}, \eqref{d3}, and~\eqref{d4}, we find that
\begin{equation*}
	\limsup_{j\to\infty} \int_0^1 \|(u_j-v_j)(\tau)\|_1\ \mathrm{d}\tau \le \varepsilon |\Omega| \max\{ M , \langle v^{in} \rangle\}
\end{equation*}
for all $\varepsilon\in (0,1)$. Therefore,
\begin{equation*}
	\lim_{j\to\infty} \int_0^1 \|(u_j-v_j)(\tau)\|_1\ \mathrm{d}\tau = 0,
\end{equation*}
and, recalling~\eqref{b18}, we have established that
\begin{equation}
	\lim_{j\to\infty} \int_0^1 \|u_j(\tau) - v_\infty\|_1\ \mathrm{d}\tau = 0. \label{d5}
\end{equation}
We now combine~\eqref{b18.2} and~\eqref{d5} to obtain that, for $\vartheta\in C_c^\infty(\Omega)$, 
\begin{align*}
	\int_\Omega (v_\infty-M)\vartheta\ \mathrm{d}x & = \int_0^1 \int_\Omega (v_\infty-M)\vartheta\ ~\mathrm{d}x ~\mathrm{d}\tau = \lim_{j\to\infty} \int_0^1 \int_\Omega (u_j(\tau)-M)\vartheta\ \mathrm{d}x \\
	& = - \lim_{j\to\infty} \int_0^1 \int_\Omega P_j(\tau) \Delta\vartheta\ ~\mathrm{d}x ~\mathrm{d}\tau = - \int_0^1 \int_\Omega \big( Q_\infty + v_\infty - M \big) \Delta\vartheta\ ~\mathrm{d}x ~\mathrm{d}\tau \\
	& = \int_\Omega (v_\infty - M) \vartheta\ \mathrm{d}x - \int_\Omega (v_\infty-M) \Delta\vartheta\ \mathrm{d}x.
\end{align*}
Therefore,
\begin{equation*}
	- \int_\Omega Q_\infty \vartheta\ \mathrm{d}x = \int_\Omega (v_\infty - M) \Delta \vartheta\ \mathrm{d}x = 0, \qquad \vartheta\in C_c^\infty(\Omega),
\end{equation*}
a property which implies that $Q_\infty=0$ and thus $v_\infty=M$. As in the proof of Theorem~\ref{thm1}, we have thus shown that the constant state $M$ is the only possible cluster point of the family $\mathcal{V} = \{ v(t)\ :\ t\ge 0\}$ in $L^1(\Omega)$ as $t\to\infty$, from which the convergence~\eqref{b28} follows. An immediate consequence is that~\eqref{b20} holds true with $t_j=j$ and we infer from~\eqref{d5} that
\begin{equation*}
	\lim_{j\to\infty} \int_j^{j+1} \|u(t)-M\|_1\ \mathrm{d}t = 0,
\end{equation*}
which completes the proof.
\end{proof}

\begin{proof}[Proof of Theorem~\ref{thm3}]
Again, we start by proving that the only cluster point of $\mathcal{V}$ in $L^1(\Omega)$ as $t\to\infty$ is $M$. We first use assumption~\eqref{gac1}, together with~\eqref{b15a}, to obtain that, for any $\theta'\ge0$,
\begin{equation} \label{alt3}
\sup_{t\in [0,\infty)} \int_\Omega \left( \frac{(1+v)^{\theta'}}{\gamma(v)}\right)^{\frac2{\aone+\theta'}} \mathrm{d}x  \le 2 K_\gamma^{\frac2{\aone+\theta'}} \left( |\Omega| + C_3^2 \right).
\end{equation}
In view of assumption~\eqref{nd4} we now split the analysis into two cases, depending on how close $u$ and $v$ are.

On the set $\left\{\frac12 v < u < 2v\right\}$, we use assumption~\eqref{nd4}, together with~\eqref{b15b}, and deduce that
\begin{equation}\begin{split}\label{bet3}
\int_0^\infty \int_\Omega & \mathbf{1}_{(\frac12,2)}  \left(\frac uv\right)  \frac{|u-v|^{\atwo+1}}{(1+v)^{\theta}} \gamma(v) \ ~\mathrm{d}x ~\mathrm{d}t \\
\le & 2^\theta \int_0^\infty \int_\Omega \mathbf{1}_{(\frac12,2)}  \left(\frac uv\right) \frac{|u-v|^{\atwo+1}}{(1+\max\{u,v\})^{\theta}}  \gamma(v) \ ~\mathrm{d}x ~\mathrm{d}t \le 2^\theta \frac{C_3}{K_\phi}.
\end{split}\end{equation}

Let $p_\atwo=2 \frac{\atwo+1}{2+\aone+\theta}\ge1$. We use H\"older's inequality and \eqref{alt3} with $\theta'=\theta$ to get
\begin{align*}
	& \left\|(u-v) \mathbf{1}_{(\frac12,2)} \left(\frac uv\right)\right\|_{p_\atwo}^{p_\atwo}\\
	& \quad = \int_\Omega \mathbf{1}_{(\frac12,2)} \left(\frac uv\right) |u-v|^{p_\atwo} \left(\gamma(v)(1+v)^{-\theta} \right)^{p_\atwo/(\atwo+1)} \left(\gamma(v)(1+v)^{-\theta} \right)^{-p_\atwo/(\atwo+1)}\ \mathrm{d}x \\
	& \quad \le \left( \int_\Omega  \mathbf{1}_{(\frac12,2)} \left(\frac uv\right) \frac{|u-v|^{\atwo+1}}{(1+v)^{\theta}} \gamma(v) \ \mathrm{d}x \right)^{p_\atwo/(\atwo+1)} \left( \int_\Omega \left(\frac{(1+v)^{\theta}}{\gamma(v)}\right)^{p_\atwo/(\atwo+1-p_\atwo)}\ \mathrm{d}x \right)^{(\atwo+1-p_\atwo)/(\atwo+1)} \\
	& \quad \le 2^{\frac{\aone+\theta}{2+\aone+\theta}} K_\gamma^{\frac{2}{2+\aone+\theta}} \left( |\Omega| + C_3^2 \right)^{\frac{\aone+\theta}{2+\aone+\theta}} \left( \int_\Omega  \mathbf{1}_{(\frac12,2)} \left(\frac uv\right) \frac{|u-v|^{\atwo+1}}{(1+v)^{\theta}} \gamma(v)\ \mathrm{d}x \right)^{2/(2+\aone+\theta)}.
\end{align*}
Combining this last inequality with \eqref{bet3}, we finally get
\begin{equation}
	\int_0^\infty \left\|(u-v) \mathbf{1}_{(\frac12,2)} \left(\frac uv\right)\right\|_{p_\atwo}^{\atwo +1} \mathrm{d}t
	 \le C_5:= 2^{\frac{\aone+3\theta}2} K_\gamma \left( |\Omega| + C_3^2 \right)^{\frac{\aone+\theta}2} \frac{C_3}{K_\phi}. \label{cvset1}
\end{equation}

On the set $\left\{\frac12 v \ge u\} \cup \{u \ge 2v\right\}$, assumption~\eqref{nd4}, together with~\eqref{b15b}, gives
\begin{equation*}
	\begin{split}
	\int_0^\infty \int_\Omega \mathbf{1}_{[0,\frac12]\cup[2,\infty)}  \left(\frac uv\right) |u-v|^{m+1} \gamma(v) \ ~\mathrm{d}x ~\mathrm{d}t \le \frac{C_3}{K_\phi}.
	\end{split}
\end{equation*}

Then, setting $p_m:= 2\frac{m+1}{2+\aone}\ge1$, similar computations as before with $\theta$ and $\atwo$ replaced by $0$ and $m$, respectively, give
\begin{align} \label{cvset2}
	\int_0^\infty \left\|(u-v) \mathbf{1}_{[0,\frac12]\cup[2,\infty)} \left(\frac uv\right)\right\|_{p_m}^{m +1} \mathrm{d}t
	 \le C_6:= 2^{\frac{k}{2}} K_\gamma \left( |\Omega| + C_3^2 \right)^{\frac{k}{2}} \frac{C_3}{K_\phi}.
\end{align}

Finally, let $p\in\big[1,\min\{p_\atwo,p_m\}\big]\subset \big[1,\min\{\atwo+1,m+1\}\big]$. Thanks to H\"older's inequality, applied first over the domain $\Omega$ then over $(t_j,t_j+1)$, we have, for $j\ge 1$,
\begin{align*}
	& \int_0^1 \|(u_j-v_j)(\tau)\|_{p}^{p}\ \mathrm{d}\tau \\
	& \qquad = \int_{t_j}^{1+t_j} \left\|\mathbf{1}_{(\frac12,2)} \left(\frac uv\right) (u-v)\right\|_{p}^{p}\ \mathrm{d}t + \int_{t_j}^{1+t_j} \left\|\mathbf{1}_{[0,\frac12]\cup[2,\infty)} \left(\frac uv\right) (u-v)\right\|_{p}^{p}\ \mathrm{d}t\\
	& \qquad \le |\Omega|^{\frac{p_\atwo -p}{p_\atwo}} \left(\int_{t_j}^{1+t_j} \left\|\mathbf{1}_{(\frac12,2)} \left(\frac uv\right) (u-v)(t)\right\|_{p_\atwo}^{\atwo+1}\ \mathrm{d}t\right)^\frac{p}{\atwo+1} \\
	& \qquad \qquad \qquad \qquad \qquad \qquad+ |\Omega|^{\frac{p_m-p}{p_m}}  \left(\int_{t_j}^{1+t_j} \left\|\mathbf{1}_{[0,\frac12]\cup[2,\infty]} \left(\frac uv\right) (u-v)(t)\right\|_{p_m}^{m+1}\ \mathrm{d}t \right)^\frac{p}{m+1}, 
\end{align*}
so that~\eqref{cvset1} and~\eqref{cvset2} imply
\begin{equation}
	\lim_{j\to\infty} \int_0^1 \|(u_j-v_j)(\tau)\|_{p}^{p}\ \mathrm{d}\tau = 0. \label{c20}
\end{equation}
Recalling~\eqref{b18}, we deduce from~\eqref{c20} that~\eqref{d5} holds, and we can proceed as in the proof of Theorem \ref{thm2} to prove \eqref{b28}.

Finally, for $t\ge 0$,
\begin{equation*}
	\int_{t}^{t+1} \|u(\tau)-M\|_{p}^{p}\ \mathrm{d}\tau \le 2^{p-1} \int_{t}^{t+1} \|(u-v)(\tau)\|_{p}^{p}\ \mathrm{d}\tau
	+ 2^{p-1} \sup_{\tau\in [t,t+1]}\{\|v(\tau)-M\|_{p}^{p}\}.
\end{equation*}
Thus, assuming further that $p<2$, we may pass to the limit as $t\to\infty$ in the above inequality with the help of~\eqref{b28} and~\eqref{c20} and find
\begin{equation*}
	\lim_{t\to\infty}\int_{t}^{t+1} \|u(\tau)-M\|_{p}^{p}\ \mathrm{d}\tau = 0,
\end{equation*}
thus completing the proof.
\end{proof}

\section{Application to $\phi$ and $\gamma$ satisfying~\eqref{pgck}} \label{sec:appli}

We now come back to the case where $\phi$ and $\gamma$ are given by~\eqref{pgck} and satisfy $(\phi\gamma)'\ge 0$. We first prove Theorem~\ref{thm0} and then shortly comment its outcome.

\begin{proof}[Proof of Theorem~\ref{thm0}]
We first note that $\phi(s)=s^m$, $s\ge 0$, satisfies \eqref{nd2} with $\eta_\lambda=\lambda^m$ so that, in the case where $\inf_\Omega v^{in}>0$, a direct application of Theorem~\ref{thm2} (together with Remarks~\ref{rem2b} and~\ref{rem2}) yields~\eqref{cv0} for all $q\in[1,2)$ and for $p=1$.

Let us now suppose that $s_0>0$ and $m\ge \frac{k}{2}>0$. We verify that $\phi(s)=s^m$ satisfies the condition~\eqref{nd4} for any $\atwo\ge1$ with $\atwo\ge m$ and for $\theta=\atwo-m\ge0$. For instance, one can take $m=\atwo$ and $\theta=0$ when $m\ge1$, while one may take $\atwo=1$ and $\theta = 1-m>0$ when $m<1$. In both cases, the convergence~\eqref{cv0} holds for any $q\in[1,2)$ and any $p\in\left [1,2\frac{m+1}{\aone+2}\right]$ with $p<2$.


Finally, by~\eqref{pgck}, we have that
	\begin{equation*}
		(\phi\gamma)(s) = s^m \left(a+\frac{b}{(s+s_0)^k}\right) \ge \frac{b}{2^k} s^{m-k}, \qquad s\ge s_0,
	\end{equation*}
	so that, if $m+1>k$, then
	\begin{equation*}
		\Psi(s) \ge \Psi(s_0) + \frac{b}{2^k} \frac{s^{m+1-k} - s_0^{m+1-k}}{m+1-k}, \qquad s\ge s_0.
	\end{equation*}
	We then infer from~\eqref{b15a} and the above lower bound on $\Psi$ that $\mathcal{V} = \{v(t)\ :\ t\ge 0\}$ is bounded in $L^{\max\{2,m+1-k\}}(\Omega)$. Consequently, if $m>k+1$, then the convergence~\eqref{cv0} can be extended to all $q\in [2,m+1-k)$.
\end{proof}

We conclude with a short discussion of our results in light of the analysis performed in \cite{ChKi2023}. For $\phi$ and $\gamma$ satisfying~\eqref{pgck}, Choi \& Kim show by a linear stability analysis \cite[Theorem~3.2]{ChKi2023}, supplemented with numerical tests, that the emergence of pattern formation is connected with the condition that the initial (averaged) mass $M$ of cells, see~\eqref{def:mass}, belongs to the set $E$ defined in~\eqref{def:E}. As a first consequence, no pattern formation is expected when the set $E$ is empty, a condition that can be rewritten as
\begin{equation}\label{pgmon}
(\phi \gamma)'(s)\ge 0, \qquad s>0.
\end{equation}
On the one hand, assuming positive values for the initial concentration $v^{in}$, Theorem~\ref{thm0} gives the asymptotic spatial homogeneity exactly under the condition~\eqref{pgmon}. In this sense, our result appears to be optimal in this case. On the other hand, when we allow the initial concentration $v^{in}$ to vanish, Theorem~\ref{thm0} requires the supplementary constraints that $s_0>0$ and $m\ge\frac{k}{2}$. The constraint $s_0>0$ is crucial to avoid a singularity of $\gamma$ at $s=0$ and also appears in the linear stability analysis \cite[Theorem~3.2]{ChKi2023}. As for the other constraint $m\ge\frac{k}{2}$, let us first point out that, when $a=0$, 
the condition~\eqref{pgmon} exactly rewrites as $m \ge k$, which is a stronger condition than $m\ge\frac{k}{2}$. However, when $a>0$ and $s_0>0$, the condition~\eqref{pgmon} becomes more intricate and actually involves $m$, $k$, $\frac{a}{b}$, and $s_0$. This condition is compatible with $m < \frac{k}{2}$, as the following computation shows: since 
\begin{equation*}
(\phi \gamma)'(s) = \frac{s^{m-1}}{(s+s_0)^{k+1}} \left[ am (s+s_0)^{k+1} + b(m-k)s + bms_0 \right], \qquad s>0,
\end{equation*}
it satisfies
\begin{equation*}
	(\phi\gamma)'(s)\ge 0 \;\; {\text{ for }\; s > 0} \;\; \text{ if and only if } \;\; \frac{a}{b} s_0^k \ge \frac{1}{m} \left( \frac{(k-m)_+}{(k+1)} \right)^{k+1}.
\end{equation*}
While the latter is obviously true when $m\ge k$, it also holds true for $m\in (0,k)$ provided $\frac{a}{b} s_0^k$ is large enough. For such a choice, we expect no pattern formation according to the analysis of \cite{ChKi2023} but Theorem~\ref{thm0} proves it only for $v^{in}>0$ when $m\in \left( 0,\frac{k}{2}\right)$. We may nevertheless apply Theorem~\ref{thm1} in this range of $m$ and obtain the asymptotic spatial homogeneity in a very weak sense, see~\eqref{cv1}, even without the initial positivity of $v^{in}$. This refinement holds whenever $(\phi \gamma)'$ is positive on $(0,\infty)$. As a conclusion, the class of functions $\phi$ and $\gamma$ given by~\eqref{pgck} for which we expect the asymptotic homogeneity but cannot prove it, even in a very weak sense, is the quite specific class of functions~\eqref{pgck} for which $\phi \gamma$ is non-decreasing with $m\in\left(0,\frac{k}{2}\right)$ and has at least one critical point. This includes the case when $(a,b,s_0,k)\in (0,\infty)^4$ and $m\in \left(0,\frac{k}{2}\right)$ satisfies
\begin{equation}
	a m (k+1)^{k+1} s_0^k= b(k-m)^{k+1}. \label{z}
\end{equation}
Indeed, introducing $s_1 := \frac{m+1}{k-m} s_0$, we infer from~\eqref{z} that
\begin{align*}
	& am(k+1) (s_1+s_0)^k = b(k-m)\\
	& am(s_1+s_0)^{k+1} + b(m-k) s_1 + bm s_0 = 0,
\end{align*}
which implies, together with the convexity of $s\mapsto s^{k+1}$, that, for $s>0$,
\begin{align*}
	(\phi \gamma)'(s) & = \frac{s^{m-1}}{(s+s_0)^{k+1}} \left[ am (s+s_0)^{k+1} + b(m-k)s + bms_0 \right] \\
	& = \frac{s^{m-1}}{(s+s_0)^{k+1}} \left[ am (s+s_0)^{k+1} - am (s_1+s_0)^{k+1} + b(m-k)(s-s_1) \right] \\
	& \qquad + \frac{s^{m-1}}{(s+s_0)^{k+1}} \left[ am (s_1+s_0)^{k+1} + b(m-k)s_1 + bms_0 \right] \\
	& = \frac{s^{m-1}}{(s+s_0)^{k+1}} \left[ am (s+s_0)^{k+1} - am (s_1+s_0)^{k+1} - am(k+1) (s_1+s_0)^k (s-s_1) \right]\ge 0 
\end{align*}
and $(\phi\gamma)'(s_1)=0$.

Finally, a question raised by the analysis in \cite{ChKi2023} that we leave open is to understand the long-time behaviour when $E$ is not empty, and, in particular, to show the expected asymptotic spatial homogeneity when $M\notin E$. Since the strategy of proof presented in this paper relies on estimates that are global in time and space, and makes full use of the monotonicity of $\phi \gamma$ everywhere, answering this last question would require a completely new strategy, and might require more refined local estimates.

\section*{Acknowledgments}

Part of this work was done while we enjoyed the hospitality of KAIST, Daejeon. We thank Yong-Jung Kim and Kyunghan Choi for stimulating discussions.

\bibliographystyle{siam}
\bibliography{KoreanAirline}

\end{document}